\documentclass[10pt]{article}
\usepackage{amssymb,kotex}
\usepackage{amsmath}
\usepackage{epsfig}
\usepackage{graphicx}
\usepackage{subcaption}
\usepackage{algorithm}
\usepackage{verbatim}
\usepackage{algpseudocode}
\usepackage[all]{xy}
\usepackage{color}
\usepackage{enumitem}  
\usepackage[color=green!40]{todonotes}
\usepackage{setspace}
\usepackage{amsthm,amsfonts,amssymb,epsfig,graphics,amsmath,amsbsy}%,showkeys,amscd}
\begin{document}
\newcommand{\intL}{\int\limits}
\newcommand{\half}{^\infty_0 }
\newcommand{\intR}{\int\limits_{\mathbb{R}} }
\newcommand{\intRR}{\int\limits_{\mathbb{R}^2} }
\newcommand{\RR}{\mathbb{R}}
\newcommand{\xx}{\mathbf{x}}
\newcommand{\uu}{\mathbf{u}}
\newcommand{\yy}{\mathbf{y}}
\newcommand{\zz}{\mathbf{z}}
\newcommand{\aaa}{\mathbf{a}}
\newcommand{\ttheta}{\boldsymbol{\theta}}
\newcommand{\pphi}{\boldsymbol{\phi}}
\newcommand{\oomega}{\boldsymbol{\omega}}
\newcommand{\xxi}{\boldsymbol{\xi}}
\newcommand{\rmd}[1]{\mathrm d#1}
\newcommand{\ub}{\textbf{\textit{u}}}
\newcommand{\xb}{\textbf{\textit{x}}}
\newcommand{\kkappa}{\boldsymbol{\kappa}}
\newtheorem{thm}{Theorem}
\newtheorem{defi}[thm]{Definition}
\newtheorem{rmk}[thm]{Remark}
\newtheorem{cor}[thm]{Corollary}
\newtheorem{lem}[thm]{Lemma}
\newtheorem{prop}[thm]{Proposition}

%\modulolinenumbers[5]

%\journal{Applied Mathematics and Computation}
%\bibliographystyle{elsarticle-num}
\title{Singular value decomposition of the wave forward operator with radial variable coefficients} 
%\date{}
\author{Minam Moon$^{a}$, Injo Hur$^{b}$, Sunghwan Moon$^{c}$}%\fnref{myfootnote}}
\date{$^{a}$ Department of Mathematics, Korea Military Academy, Seoul 01805, Republic of Korea\\
$^{b}$ Department of Mathematical Education, Chonnam National University, Gwangju 61186, Republic of Korea\\
$^{c}$ Department of Mathematics, College of Natural Sciences, Kyungpook National University, Daegu 41566, Republic of Korea}
%\fntext[myfootnote]{shmoon@unist.ac.kr}
\maketitle
%
%%\begin{frontmatter}
 \begin{abstract}
Photoacoustic tomography (PAT) is a novel and promising technology in hybrid medical imaging that involves generating acoustic waves in the object of interest by stimulating electromagnetic energy.
The acoustic wave is measured outside the object. One of the key mathematical problems in PAT is the reconstruction of the initial function that contains diagnostic information from the solution of the wave equation on the surface of the acoustic transducers.
Herein, we propose a wave forward operator that assigns an initial function to obtain the solution of the wave equation on a unit sphere.
Under the assumption of the radial variable speed of ultrasound, we obtain the singular value decomposition of this wave forward operator by determining the orthonormal basis of a certain Hilbert space comprising eigenfunctions.
In addition, numerical simulation results obtained using the continuous Galerkin method are utilized to validate the inversion resulting from the singular value decomposition.

\texttt{Keywords: photoacoustic; tomography; image reconstruction; wave equation; singular value decomposition; Helmholtz equation}\\
\texttt{MSC2010:	35L05; 34B24; 31B20; 35R30; 	92C55}
 \end{abstract}

%%%%%%%%%%%%%%%%%%%%%%%%%%%%%%%%%%%%
 \section{Introduction}
%%&%%%%%%%%%%%%%%%%%%%%%%%%

Hybrid tomographic techniques, which use two different physical signals to obtain enhanced images, have been extensively studied.
Photoacoustic tomography (PAT) is the most successful example of hybrid biomedical imaging that is based on the photoacoustic effect discovered by Bell \cite{bell80}. It offers the advantages of both pure optical and ultrasound imaging.
Pure ultrasound imaging typically provides high-resolution images with a low contrast between the cancerous and healthy tissues, whereas optical or radio-frequency electromagnetic imaging offers high contrast with low resolution. 
PAT incorporates ultrasound and optical or radio-frequency electromagnetic waves and provides high-contrast and high-resolution images.

In PAT, the object of interest is irradiated by pulsed nonionizing electromagnetic energy, causing a small level of heating in the interior. 
A pressure wave is generated by the resulting thermoelastic expansion and propagates through the object. 
The electromagnetic energy absorbed is significantly higher in cancerous cells than in healthy tissues, and this absorbed energy represents the initial pressure.
Specifically, the initial pressure $f$ contains highly useful diagnostic information.
The pressure $p$ is measured using acoustic transducers placed along a surface completely (or partially) surrounding the object (see \cite{ammaribjk10,ammarigjn13,kuchment14book,kuchmentk08,xuw06} or references therein).

In this section, we describe the mathematical model underlying PAT and address the key mathematical problems associated with it.
We assume that point-like broadband ultrasound transducers are located along the observation surface $\Gamma$. 
Thus, the measured data represent the values of pressure $p(\xx,t)$ along the observation surface $\Gamma$, that is, $p(\xx,t)|_{\xx\in S}$.
It is assumed that $\Gamma$ is a unit sphere $S^{n-1}$ and 
the object of interest is inside a unit ball $B$.
With the speed of sound $c(\xx)$ at a location $\xx$, the following model is usually used to describe the propagating pressure wave $p(\xx,t)$ generated in PAT for $n=2,3$:
\begin{equation}\label{eq:pdeofpatorgin}
\begin{array}{ll}
\partial_t^2 p(\xx ,t)=c(\xx)\triangle_{\xx }p(\xx ,t)\qquad(\xx ,t)\in B\times[0,\infty)\\
p(\xx ,0)=f(\xx ) \quad\mbox{and}\quad\partial _t p(\xx ,0)=0 \\
\partial_\nu p(\xx,t)=0\quad\mbox{on}\quad\xx\in S^{n-1}\times[0,\infty).
\end{array}
\end{equation}
Here $f(\xx)$ is the required PAT image and $\nu$ denotes the outward normal to $S^{n-1}$ (\cite{ammarikk12,ammaridkk15}).

One of the mathematical problems associated with PAT is the recovery of the initial function $f$ from the solution of the wave equation on $S^{n-1}$.
We define the wave forward operator as $\mathcal W f(\xx,t)=p(\xx,t),$ $(\xx,t)\in B\times [0,\infty)$. 
%More precisely, we extend $f\in L^2(B, c(|\xx|)^{1-n})$ to be a continuous function on the whole space $\RR^n$ that vanishes outside the unit ball $B$.
%Then $\mathcal W f(\ttheta,t)$ satisfies \eqref{eq:pdeofpatorgin} with the extened function of $f$. 

There are several methods for reconstructing $f$ from $\mathcal Wf|_{S^{n-1}\times[0,\infty)}$ at constant speed \cite{ammaribjk10,ammarigjn13} (several studies \cite{anastasiozmr07,moonip18,dreierh20,finchhr07,finchpr04,moonjop16,kostlifbw01,kunyansky12,moonjmaa18,natterer12,nguyen09,sandbichlerkbbh15,moonzangerlip19} have focused on this topic, but $
\mathcal Wf $ satisfies the wave equation on the entire space without a boundary condition).
In this study, we obtain the singular value decomposition (SVD) of $\mathcal W$ under the assumption that the speed of sound is a radial function, that is, $c(|\xx|)$.
It is also reasonable to assume that the speed of sound is strictly positive, bounded away from $0$ and above, that is, $0<c_m< c(|\xx|)< c_M$ for positive constants $c_m$ and $c_M$ \cite{agranovskyk07,ammaridkk15,ammarikk12,stefanovu09}. %
%Here we show how to reconstruct $f$ from $\mathcal Wf$ through SVD of $\mathcal W$. To the best of our knowledge, the exact inversion of $\mathcal W$ with radial variable coefficient is first derived.

SVD is a powerful tool used to analyze a compact operator \cite{kazantsev15} and provides considerable insight into characterizing the range of the operator and inverting it for the corresponding inverse problems \cite{louis84,natterer01}.
Although the SVD of a spherical Radon transform has been studied previously \cite{moonip20,quinto83}, to the best of our knowledge, this is the first study focusing on the SVD of the wave forward operator.

In Section \ref{sec:SVD}, we present the SVD of the wave forward operator. To obtain the SVD, we find an orthonormal basis of $L^2(B, c(|\xx|)^{1-n})$ consisting of the eigenfunctions of the following based on the Sturm--Liouville problem: 
 $$
 c(|\xx|)\triangle_\xx\phi(\xx)+\mu^2_k\phi(\xx)=0 \quad\mbox{ and }\quad\partial_\nu \phi|_{S^{n-1}}=0.
$$
 Section \ref{sec:numerical} presents numerical simulations.
The continuous Galerkin finite element method (CG FEM) is used to determine the orthonormal basis $\{\phi\}$ and demonstrate the validity of the inversion formula derived from the SVD.
This paper ends with a discussion of the wave forward operator with the Dirichlet boundary condition in Section \ref{sec:discuss}.

%
%%%%%%%%%%%%%%%%%%%%%%%%%%%%%%%%%%%%%%%%%%%
\subsection{Preliminaries}
%%%%%%%%%%%%%%%%%%%%%%%%%%%%%%%%%
In this subsection, we introduce a certain non-separable Hilbert space on $S^{n-1}\times [0,\infty)$ that becomes the codomain of the wave forward operator. 
First, for  two functions $g_1$ and $g_2$ on $[0,\infty)$,  
$$
\langle g_1,g_2\rangle_{H[0,\infty)}=\lim_{A\to\infty}\frac2{A}\intL_{0}^A g_1(t)\overline{g_2(t)}{\rm d}t
$$
is defined as an inner product that satisfies the following property:
\begin{prop}\label{prop:hilbert1}
For $\iota,\iota'>0$,
\begin{equation}\label{eq:expansionSVDH1}
\langle \cos(\iota\cdot),\cos(\iota' \cdot) \rangle_{H[0,\infty)}=\delta_{\iota,\iota'}.
\end{equation}
\end{prop}
\begin{proof}
Direct computation yields 
$$
\begin{array}{ll}
\langle \cos(\iota \cdot),\cos(\iota \cdot) \rangle_{H[0,\infty)}=\displaystyle\lim_{A\to \infty}\frac2{A}\intL_{0}^A \cos^2(\iota t){\rm d}t=\lim_{A\to \infty}\left(1+\frac{\sin(2A \iota)}{2 A\iota}\right)=1,
\end{array}
$$
and for $\iota\neq\iota'$
$$
\begin{array}{ll}
\langle \cos(\iota \cdot),\cos(\iota' \cdot) \rangle_{H[0,\infty)}\displaystyle=\lim_{A\to \infty}\frac2{A}\intL_{0}^A \cos(\iota t)\cos(\iota' t){\rm d}t
=\lim_{A\to \infty}\frac{2\iota'\cos(A\iota)\sin(A\iota')-2\iota\cos(A\iota')\sin(A\iota)}{(\iota'^2-\iota^2)A}=0.
\end{array}
$$
\end{proof}
For $(l,\iota)\in\mathbb Z\times [0,\infty)$, let 
$$
\Psi_{l,\iota}(\ttheta,t)=\frac{e^{\mathrm{i}l\theta}\cos(\iota t)}{\sqrt{2\pi}} \quad\mbox{for}\quad (\ttheta,t)\in S^1\times [0,\infty),
$$
and $X(S^1\times[0,\infty))$ be the complex vector space comprising all finite linear combinations of these functions $\Psi_{l,\iota}$.
For $f,g\in X(S^1\times[0,\infty))$ and $a>0$,
$$
\langle f,g \rangle_H=\lim_{A\to\infty}\frac2{A}\intL_{0}^A \intL_{S^1}f(\ttheta,t)\overline{g(\ttheta,t)}{\rm d}S(\ttheta){\rm d}t
$$
is an inner product, which transforms $X(S^1\times[0,\infty))$ into a unitary space. Here, $\bar z$
is a complex conjugate of $z\in\mathbb C$. The completion of this space is a non-separable Hilbert space $H(S^1\times[0,\infty))$ such that $\{\Psi_{l,\iota}\}_{(l,\iota)\in \mathbb Z\times [0,\infty)}$ is an orthonormal basis in $H(S^1\times[0,\infty))$ (see \cite{rudin87}).
%Let $\{\psi_k(\xx)\}$ be the  eigenfunctions of $-c(|\xx|)\triangle_\xx$ with eigenvalue $\mu_k^2$ and $ \partial_\nu   \psi_k|_{ S^1}=0$, i.e., 
%\begin{equation}\label{eq:helmholz}
%c(|\xx|)\triangle_\xx\psi_k(\xx)+\mu^2_k\psi_k(\xx)=0\quad\mbox{on}\quad B.
%\end{equation}
%Then $ \{\psi_k(\xx)\}_{k=1}^\infty$ is the orthonormal basis of $L^2( B,c(\xx)^{-1}{\rm d}\xx)$ and thus
%$$
%f=\sum_{k=1}^\infty <f,\psi_k>\psi_k \quad\mbox{where}\quad<f,g>=\intL_{ B}f(\xx)\overline{g(\xx)}c(\xx)^{-1}{\rm d}\xx.
%$$
%Then we can easily to check $\mathcal W_{}f(\xx,t)=\sum_{k=1}^\infty <f,\psi_k>\psi_k(\xx)\cos(\mu_kt)$.  

To obtain the SVD of the wave forward operator, we determine the eigenfunctions of 
 $$
 c(|\xx|)\triangle_\xx\phi(\xx)+\mu^2_k\phi(\xx)=0
\quad\xx\in B\quad\mbox{ and }\quad\partial_\nu \phi|_{S^{n-1}}=0,
$$
which form an orthonormal basis for $L^2(B, c(|\xx|)^{1-n})$, the  $L^2$-space  defined on $B$ with a weight $c(|\xx|)^{1-n}$.

%%%%%%%%%%%%%%%%%%%%%%%%
\section{SVD of the wave forward operator}\label{sec:SVD}
%%%%%%%%%%%%%%%%%%%%%
In this section, we present the SVD of the wave forward operator $\mathcal Wf$ with a radial variable speed.
%%%%%%%%%%%%%%%%%%%%%%%%%%%%%%%%
\subsection{Two dimensions}
%%%%%%%%%%%%%%%%%%%%%%%
For $l\in\mathbb Z$, let $\phi_{k,l}(\xx)=h_{k,l}(|\xx|)e^{\mathrm{i}l\theta_\xx}/\sqrt{2\pi}$ and $\tan\theta_\xx=x_1/x_2.$
As $\triangle_\xx\phi =\partial_{r_{}}^2\phi+r_{ }^{-1}\partial_{r_{}}\phi+r_{ }^{-2}\partial_{\theta_{\xx}}^2\phi$, 
\begin{equation}\label{eq:helmholtz}
c(|\xx|)\triangle_\xx\phi_k(\xx)+\mu^2_k\phi_k(\xx)=0\quad\mbox{on}\quad 
  \xx=(r\cos\theta_{\xx},r\sin\theta_{\xx})  \in B \subset\RR^2 
\end{equation}
can be transformed into 
$$
    c(r)\left(\partial_{r}^2h_{k,l}(r)e^{\mathrm{i}l\theta_{\xx}}+r_{}^{-1}\partial_{r}h_{k,l}(r)e^{\mathrm{i}l\theta_{\xx}}-r_{}^{-2}l^2h_{k,l}(r)e^{\mathrm{i}l\theta_{\xx}}\right)+\mu^2_{k,l}h_{k,l}(r)e^{\mathrm{i}l\theta_{\xx}}=0,
$$
or equivalently
\begin{equation}\label{eq:ode}
r_{}^2\frac{\rm d^2}{{\rm d}r^2}h_{k,l}(r_{})+r_{{}}^{}\frac{\rm d}{{\rm d}r}h_{k,l}(r_{})+\left(\frac{r_{}^2\mu_{k,l}^2}{c(r_{})}-l^2\right)h_{k,l}(r_{})=0.
\end{equation}
After removing the indices for convenience, \eqref{eq:ode} can be written as
\begin{equation}\label{eveq:ode}
 \mathcal   S(h)(r)= \mu^2 r c(r)^{-1} h(r), 
\end{equation}
where 
\begin{equation*}
   \mathcal   S(h)(r):=-\frac{\rm d}{{\rm d}r} \left(r\frac{\rm d}{{\rm d}r} h(r) \right)+\frac{l^2}{r}h(r).%  \quad \left(\textrm{here } ':=\frac{\rm d}{{\rm d}r}\right). 
\end{equation*}
It must be noted that $ \mathcal S$ is the Sturm--Liouville operator acting on a weighted $L^2$ space, $L^2((0,1], rc(r)^{-1})$, and \eqref{eveq:ode} is its eigenvalue equation. Therefore, we can apply the theory of Sturm--Liouville operators. To ensure that the manuscript is concise as well as self-contained, the method of application is described in Appendix A of this paper. Using the method presented in Appendix A and applying the boundary condition, we obtain 
\begin{equation}\label{bc at 1}
\frac{\rm d}{{\rm d}r} h(1)=0
\end{equation}
at the endpoint $r=1$. Then $\mathcal S$ becomes a self-adjoint operator on $L^2((0,1], rc(r)^{-1})$ such that its spectrum is purely discrete (i.e., its eigenvalues are of finite multiplicities, and no continuous spectrum is observed); therefore, the set of eigenfunctions corresponding to $\mathcal S$ forms an orthonormal basis on $L^2((0,1], rc(r)^{-1})$. 

%Now we consider the following Sturm-Liouville problem: 
%%\begin{equation}\label{eq:SLP}
%\begin{eqnarray}
%\displaystyle r\partial_{r}^2h(r)+\partial_{r}h(r)+(\frac{\mu_{k,l}^2r}{c(r)}-\frac{l^2}r)h(r)=0\quad\mbox{on}\quad (0,1]\label{eq:SLP}\\
%h(0+) \mbox{ exists and is finite,} \quad h'(1_{})=0.\label{eq:IC}
%\end{eqnarray}

%%%%%%%%%%%%%%%%%%%%%%%%
\begin{comment}
Consider \eqref{eq:T} in the following Sturm--Liouville problem with boundary condition \eqref{eq:IC}:
\begin{equation}
\displaystyle r\frac{\rm d^2}{{\rm d}r^2} h(r)+\frac{\rm d}{{\rm d}r}h(r)+\left(\frac{\mu_{k,l}^2r}{c(r)}-\frac{l^2}r\right)h(r)=0\quad\mbox{on}\quad (0,1]\label{eq:SLP}.
\end{equation}
\end{comment}
%%%%%%%%%%%%%%%%%%%%%%%%%
Let $h_{k,l}$ be the eigenfunctions of \eqref{eveq:ode} with eigenvalue $\mu_{k,l}$ such that for a fixed $l$, $\mu_{1,l}<\mu_{2,l}<\mu_{3,l}\cdots$. It is well known that the dimension of the eigenspace is one because of the separated boundary condition. The detailed description is provided in a previous study (Section 3.6 of \cite{folland92}).
Thus, 
$$
\phi_{k,l}(\xx)= h_{k,l}(|\xx|)e^{\mathrm{i}l\theta_\xx}/\sqrt{2\pi}, (k,l)\in\mathbb N\times\mathbb Z
$$
are the eigenfunctions of 
$$
c(|\xx|)\triangle_\xx\phi(\xx)+\mu^2_k\phi(\xx)=0 \quad\mbox{ and }\quad\partial_\nu \phi|_{S^{1}}=0
$$ 
with eigenvalue $\mu_{k,l}$ and they form an orthonormal basis of $L^2_{}(B,c(|\xx|)^{-1})$, because 
$$
||\phi_{k,l}||_{L^2(c(\xx)^{-1})}^2=(2\pi)^{-1}\intL_{0}^{2\pi}\intL\half |h_{k,l}(r)e^{\mathrm{i}l\theta}|^2rc(r)^{-1}{\rm d}r{\rm d}\theta=1.
$$
%%% remove . and relocate ,  
Thus, for $f\in L^2_{}(B,c(|\xx|)^{-1})$, 
$$
f(\xx)=\sum_{l\in\mathbb Z}\sum_{k=1}^\infty \langle f,\phi_{k,l}\rangle \phi_{k,l}(\xx)\quad\mbox{and}\quad \langle f,g \rangle=\intL_{ B}f(\xx)\overline{g(\xx)}c(\xx)^{-1}{\rm d}\xx.
$$
\begin{rmk}
If $h_{k,l}(1)=0$, then because of the uniqueness of the initial value problem of the second linear differential equations, the boundary condition $\frac{\rm d}{{\rm d}r} h_{k,l}(1)=0$  (from \eqref{bc at 1}) yields $h_{k,l}\equiv0$. Therefore, for nonzero values, $h_{k,l}(1)\neq0$. 
\end{rmk}
%\begin{rmk}
%If $h_{k,l}(1)=0$, $\phi_{k,l}(\ttheta)=0$ and 
%$$
%\begin{array}{ll}
%\mathcal W_{}f(\ttheta,t)&\displaystyle=\sum_{l\in\mathbb Z}\sum_{k=1}^\infty <f,\phi_{k,l}>\phi_{k,l}(\ttheta)\cos(\mu_{k,l}t)\displaystyle =\sum_{l\in\mathbb Z,k\in\mathbb N,\atop h_{k,l}(1)\neq 0} <f,\phi_{k,l}>\phi_{k,l}(\ttheta)\cos(\mu_{k,l}t).
%\end{array}
%$$ 
%\end{rmk}

%\subsection{$n=3$}
%%%%%%%%%%%%%%%%%%%%%%%%%%%%%%%%%%%%%%%%%%%%
%\section{Dirichlet Problem}\label{sec:dirichlet}
%%%%%%%%%%%%%%%%%%%%%%%%%%%%%%%%%%%%%%%%%%%%%%%%%%
For two Hilbert spaces $X$ and $Y$, let $\mathcal A: X \to Y$ be a linear operator.
The triple $\{\Phi_l,\Psi_l,\sigma_l\}_{l\ge 0}$ is an SVD of the operator $\mathcal A$ if
\begin{itemize}
	\item  $\{\Phi_l\}_l$ is an orthonormal basis in $X$;
	\item  $\{\Psi_l\}_l$ is an orthonormal set in $Y$;
	\item  $\{\sigma_l\}_l$ is a set of non-zero constants, $\mathcal A\Phi_l=\sigma_l \Psi_l.$
\end{itemize}

\begin{thm}\label{thm:co}
Let $\phi_{k,l}(\xx)= h_{k,l}(|\xx|)e^{\mathrm{i}l\theta_\xx}/\sqrt{2\pi}$, where $h_{k,l}$ is the solution of \eqref{eq:ode} (with \eqref{bc at 1}) corresponding to the eigenvalue $\mu_{k,l}$ and $||h_{k,l}||_{L^2((0,1],rc(r)^{-1})}=1$.% and $H=\span\{\phi_{k,l}:k\in\mathbb N,l\in\mathbb Z\}$.
Then, 
\begin{enumerate}
\item %the operator $\mathcal W$, densely defined on $\span\{\phi_{k,l}:k\in\mathbb N,l\in\mathbb Z\}$, uniquely extends to linear bounded operator $\mathcal W_{}\cdot|_{S^{1}\times[0,\infty)}:L^2_{}(B,c(|\xx|)^{-1})\to H(S^1\times[0,\infty))$ 
for $f\in \operatorname{span}\{\phi_{k,l}:k\in\mathbb N,l\in\mathbb Z\}$, we obtain
$$
\mathcal W_{}f(\xx,t)=\sum_{k,l} \langle f,\phi_{k,l} \rangle \phi_{k,l}(\xx)\cos(\mu_{k,l}t).
$$ 
\item $\{\phi_{k,l},\Psi_{l,\mu_{k,l}},h_{k,l}(1)\}$ is an SVD of $\mathcal W_{}\cdot|_{S^{1}\times[0,\infty)}:L^2_{}(B,c(|\xx|)^{-1})\to H(S^1\times[0,\infty))$.
\item  For $f\in \operatorname{span}\{\phi_{k,l}:k\in\mathbb N,l\in\mathbb Z\}$, we obtain 
\begin{equation}\label{eq:main}
\langle f,\phi_{k,l} \rangle = |h_{k,l}(1)|^{-2} \langle \mathcal W_{}f,\phi_{k,l}(\cdot)\cos(\mu_{k,l} \cdot)\rangle_H.
%\displaystyle\lim_{A\to\infty}\frac{1}{2\pi |h_{k,l}(1)|^2 A}\intL^{A}_0\intL_{S^1} \overline{\phi_{k,l}(\ttheta)}\cos(\mu_{k,l} t)\mathcal W_{}f(\ttheta,t) {\rm d}S(\ttheta){\rm d}t.
\end{equation}
\end{enumerate}
\end{thm}
\begin{proof}
\textit{1}. We can easily check this.%It must be noted $\mathcal W_{}\phi_{k,l}(\ttheta,t)=h_{k,l}(1)\Psi_{l,\mu_{k,l}}(\ttheta,t)$ can be easily verified, which yields \textit{1}. 

\textit{2}. It must be noted $\mathcal W_{}\phi_{k,l}(\ttheta,t)=h_{k,l}(1)\Psi_{l,\mu_{k,l}}(\ttheta,t)$ and $\{\Psi_{l,\iota}\}_{(l,\iota)\in \mathbb Z\times [0,\infty)}$ is an orthonormal basis in $H(S^1\times[0,\infty))$, as obtained in Proposition 
\ref{prop:hilbert1}.
%%%%%%%%%%% remove the and )

\textit{3}. Considering the right-hand side of \eqref{eq:main}, we obtain
$$
\begin{array}{l}
 \langle \mathcal W_{}f,\phi_{k,l}(\cdot)\cos(\mu_{k,l} \cdot)\rangle_H\\
 \qquad\qquad =\displaystyle \sum_{l',k'}\intL_{S^1} \overline{\phi_{k,l}(\ttheta)}\phi_{k',l'}(\ttheta) {\rm d}S(\ttheta) \langle f,\phi_{k',l'} \rangle  \langle\cos(\mu_{k',l'}\cdot),\cos(\mu_{k,l} \cdot)\rangle_{H[0,\infty)}\\
 \qquad\qquad=\displaystyle\sum_{l',k'} |h_{k,l}(1)|^2 \delta_{l,l'}\delta_{k,k'} \langle f,\phi_{k,l}\rangle,
\end{array}
$$
where in the first equality, we used the condition {\textit 1}, and in the second equality, we used Proposition \ref{prop:hilbert1} and the orthogonality of $e^{\mathrm{i}l\theta}$.
 
\end{proof}

%%%%%%%%%%%%%%%%%%%%%%%%%%%%%%%%%%%%%%%%%%%%%%%%%

 \subsection{Three dimensions}
In this section, we consider the case of three dimensions ($n=3$). In this case, the same procedure, as described above, is followed (except for the slightly different eigenvalue equation \eqref{eq:ode 3D}). 
 
For $l=0,1,2,\cdots $ and $ k=-l,\cdots,l$, let $\tilde\phi_{m,lk}(\xx)=\tilde h_{m,l}(|\xx|)Y_{lk}\left(\frac{\xx}{|\xx|}\right)$, where $Y_{lk}$ are spherical harmonics.
Because $\triangle_\xx\phi =r^{-2}\partial_{r_{}} (r^2\partial_{r_{}}\phi)+r_{ }^{-2}\triangle_{S^2}\phi$,  the equation 
\begin{equation*}\label{eq:helmholtz3d}
c(|\xx|)\triangle_\xx\phi_k(\xx)+\mu^2_m\phi_m(\xx)=0\quad\mbox{on}\quad B\subset\RR^3
\end{equation*}
becomes for $l\in\mathbb Z$ and $ k=-l,\cdots,l$,
$$
    \frac{c(r)}{r^{2}}\left\{\frac{\rm d}{{\rm d}r}\left(r^2\frac{\rm d}{{\rm d}r}   \tilde h_{m,l}(r)\right)Y_{lk}\left(\ttheta\right)+r_{ }^{-2}\tilde h_{m,l}(r)\triangle_{S^2}Y_{lk}\left(\ttheta\right)\right\}+\tilde\mu^2_{m,l}\tilde h_{m,l}(r)Y_{lk}\left(\ttheta\right)=0,
$$
or equivalently,
\begin{equation}\label{eq:ode 3D}
\frac{\rm d}{{\rm d}r}\left(r^2\frac{\rm d}{{\rm d}r}\tilde h_{m,l}(r_{})\right)+\left(\frac{r_{}^2\tilde\mu_{m,l}^2}{c(r_{})}-l(l+1)\right)\tilde h_{m,l}(r_{})=0,
\end{equation}
as $\triangle_{S^2}Y_{lk}\left(\ttheta\right)=-l(l+1)\triangle_{S^2}Y_{lk}\left(\ttheta\right)$.
%\begin{equation}\label{eq:ode}
%r^2\frac{\rm d^2}{{\rm d}r^2}h_{m,l}(r)+2r\frac{\rm d}{{\rm d}r}h_{m,l}(r_{})+(\frac{r_{}^2\mu_{m,l}^2}{c(r_{})}-l(l+1))h_{m,l}(r_{})=0.
%\end{equation}

Similar to the case when $n=2$ (see Remark in Appendix A), the operator related to \eqref{eq:ode 3D} with the boundary condition $\frac{\rm d}{{\rm d}r}\tilde h(1)=0$ is a self-adjoint operator and purely discrete. Therefore, the corresponding eigenfunctions form an orthonormal basis in the weighted $L^2$-space $L^2((0,1],r^2c(r)^{-1})$.  

Furthermore, similar to the $n=2$ case, we can easily verify that $\mathcal W_{}\phi_{k,l}(\ttheta,t)= \tilde h_{m,l}(1)Y_{lk}(\ttheta)\cos(\tilde\mu_{m,l}t)$, and for $f\in L^2_{}(B^1,c(|\xx|)^{-1})$, we obtain
$$
\mathcal W_{}f(\xx,t)=\sum_{k,l} \langle f,\tilde\phi_{m,lk} \rangle \tilde \phi_{m,lk}(\xx)\cos(\tilde \mu_{m,l}t)
$$ 
and 
\begin{equation*}%\label{eq:main}
\langle f,\tilde\phi_{m,lk} \rangle =\displaystyle\lim_{A\to\infty}\frac{2}{  |\tilde h_{m,l}(1)|^2 A}\intL^{A}_0\intL_{S^2} \overline{\tilde\phi_{m,lk}(\ttheta)}\cos(\tilde\mu_{m,l} t)\mathcal W_{}f(\ttheta,t) {\rm d}S(\ttheta){\rm d}t.
\end{equation*}
Finally, it can be concluded that $\{\tilde\phi_{m,l}(\xx),Y_{lk}(\ttheta)\cos(\tilde\mu_{m,l}t),\tilde h_{m,l}(1)\}$ is  the SVD of the wave forward operator $\mathcal W_{}\cdot|_{S^{3}\times[0,\infty)}:L^2_{}(B^1,c(|\xx|)^{-1})\to H(S^1\times[0,\infty))$.
\begin{comment}
For any $f\in L^2((0,1], r^2c(r)^{-1})$, the operator related to 
%there exists a unique $h\in H^1((0,1], r^2c(r)^{-1})$ satisfying
\begin{equation}\label{eq:T for 3D}
\displaystyle \frac{c(r)}{r^2}\left(-\frac{\rm d}{{\rm d}r}\left(r^2\frac{\rm d}{{\rm d}r}\right)+l(l+1)\right)h(r)=f(r)\quad\mbox{on}\quad (0,1]    
\end{equation}
with the boundary condition $h'(1)=0.$    
\end{comment}

%%%%%%%%%%%%%%%%%%%%%%%%%%%%%%%
\section{Numerical simulation}\label{sec:numerical}
%%%%%%%%%%%%%%%%%%%%%%
In this section, we present the reconstruction results in 2D with the SVD of the wave forward operator $\mathcal W_{}$ with radial variable coefficients in Theorem \ref{thm:co} using data $\mathcal W_{}f$. To numerically solve \eqref{eq:pdeofpatorgin} and \eqref{eq:helmholtz}, we consider a circular domain $B$ of radius 3 in 2D. For space discretization, we divide the circular domain $B$ into triangulations with $h_{max} = 0.15$. It contains unstructured grid points in $B$ owing to its geometric characteristics. For time discretization, we set the maximum time as $T = 800$ with a time-step size ${\rm d}t = 0.0016$. To compute the eigenfunction $\phi_{}$ of \eqref{eq:helmholtz}, we use the  continuous Galerkin (CG) FEM on the 2D circular domain $B$. The discrete problem with the CG FEM is as follows \cite{gross07, larsson03}:
\begin{equation}\label{eq:helmholtz_fem}
a(\phi, v_h) = \mu^2 m(\phi, v_h) \quad \text{for all} \quad v_h\in V_h ,
\end{equation}
where the finite element space $V_h$ is composed of piecewise linear functions and 
$$
a(u, v)=\intL_B \nabla u(\xx) \cdot \nabla v(\xx) {\rm d}\xx\quad \mbox{ and }\quad m(u, v)=\intL_B c(|\xx|)^{-1} u(\xx) v(\xx) {\rm d}\xx.
$$

To generate data from the phantom distribution $f$ and orthonormal basis set $\{\Psi_k\}$, we must numerically compute $\mathcal W_{}f$ and $\mathcal W \phi$. Therefore, we consider the following discrete wave propagation using the CG FEM with a backward-Euler scheme \cite{gross07, larsson03}:
\begin{equation}\label{eq:pat_fem}
\begin{array}{rl}
\displaystyle c(\xx)^{-1}\frac{1}{\Delta t ^2}\left(p_h^{n+1}-2p_h^n+p_h^{n-1}, v_h\right)+ a\left( p_h^{n+1}, v_h\right)&=0\\
\displaystyle p_h^0=f~ \mbox{or}~ \phi_k \quad\mbox{and}\quad
\frac{1}{\Delta t}\left(p_h^1 - p_h^{0}\right) &=0,
\end{array}
\end{equation}
for all $v_h \in V_h$.
Here, the bilinear form $a$ and finite element space $V_h$ are the same as those in \eqref{eq:helmholtz_fem}. 

\begin{figure}[h!]
\centering
\begin{subfigure}[b]{0.45\textwidth}
\includegraphics[width=\textwidth]{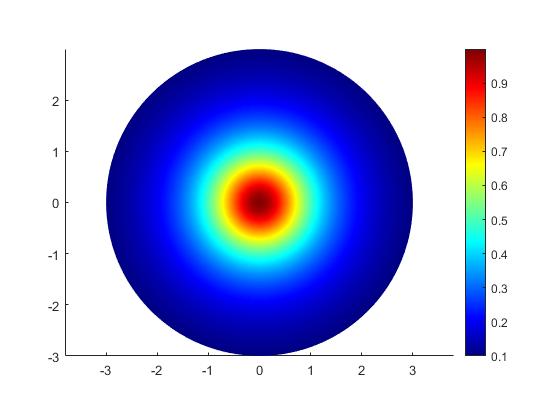}
\caption{Continuous coefficient $c_1(|\xx|)$}
\label{fig:coef-a}
\end{subfigure}
\begin{subfigure}[b]{0.45\textwidth}
\includegraphics[width=\textwidth]{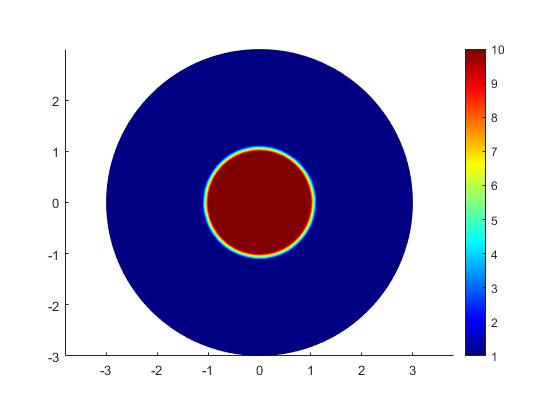}
\caption{Discontinuous coefficient $c_2(|\xx|)$}
\label{fig:coef-b}
\end{subfigure}
\caption{Examples of radial variable coefficients.}
\label{fig:coef}
\end{figure}

We tested our methods on two different distributions of the coefficient $c(|\xx|)$, as shown in Figure \ref{fig:coef}. The coefficient in Figure \ref{fig:coef}(a) is continuous and defined as $c_1(|\xx|)=1/(1+|\xx|^2)$. The coefficient in Figure \ref{fig:coef}(b) takes the values one and five   in the blue and red regions, respectively. These two coefficients reproduce the fact that the speed of the ultrasonic wave continuously decreases as it moves away from the center and it is discontinuous before and after passing through a specific object.    
      
\begin{algorithm}
\caption{SVD Reconstruction algorithm}\label{SVD algorithm}
\begin{algorithmic}[1]
\Procedure{SVD}{$f$}\Comment{The phantom distribution $f$}
   \State $\mathbf{b}\approx \mathcal Wf|_{S^{1}\times[0,\infty)}$ \Comment{Solving \eqref{eq:pat_fem} with the initial condition $f$}
   \State Find eigenfunctions $\phi_{k}$ and eigenvalue $\mu_k$ \Comment{Solving \eqref{eq:helmholtz_fem}}
   \While{$k \leq N$}\Comment{$N$ is the number of eigenfunctions}
      \State Compute $\Psi_{k} \approx \mathcal W_{}\phi_{k}|_{S^{1}\times[0,\infty)}$ \Comment{Solving \eqref{eq:pat_fem} with the initial condition $\psi_k$}
      \State $A(:,k) \gets \Psi_{k}$
   \EndWhile
   \State Solve $AX=\mathbf{b}$
   \State Reconstruction $\tilde{f} = \sum_{k=1}^{N} x_k \phi_{k}$
   \State \textbf{return} $\tilde{f}$ \Comment{Reconstruction image of $f$ is $\tilde{f}$}
\EndProcedure
\end{algorithmic}
\end{algorithm}

Algorithm \ref{SVD algorithm} describes the procedure of the image reconstruction using the SVD of the wave forward operator $\mathcal W$. 
%In the actual situation, data is given, but we use the phantom $f_1$ or $f_2$ used for simulation as the initial condition, solve \eqref{eq:pat_fem}, and then generate the data $\mathbf{b}$ restricting the value of sensing region. 
To obtain the data, we use the phantom distribution $f_1$ or $f_2$ as an initial condition, solve \eqref{eq:pat_fem}, and then generate data for $\mathbf{b}$ by restricting the value in the sensing region.
To generate the orthonormal basis $\phi_k$ of Theorem \ref{thm:co}, we solve the discrete problem in \eqref{eq:helmholtz_fem}. Additionally, we solve \eqref{eq:pat_fem} with the initial condition $\phi_k$, which is similar to the process of generating data $\mathbf{b}$ to obtain an orthonormal function $\Psi_k$ defined in the sensing region. Following the creation of a matrix using the new function $\Psi_k$ as a column vector, we solve the linear system $AX=\mathbf{b}$. Finally, the image is numerically reconstructed as a linear combination of the solution $X$ and orthonormal basis $\{\phi_k\}$.    
   
Figure \ref{fig:recon} shows the reconstruction results of the SVD algorithm described in the previous subsection applied to two different data cases $f_1$ and $f_2$ for radial variable coefficients. The first initial distribution $f_1$ can verify the discontinuous properties of an object by matching the coefficient $c_2(|\xx|)$, and the second $f_2$ can verify whether it is generally applicable to more diverse shapes. 
In Figure \ref{fig:recon}, the second column shows the reconstruction results for the continuous coefficient $c_1(|\xx|)$, and the third column shows the reconstruction results for the discontinuous coefficients $c_2(|\xx|)$. For the reconstruction of the phantom distributions $f_1$ and $f_2$, we used 1473 and 1533 orthonormal basis functions $\Psi_k$, respectively. The reconstruction results indicate that the SVD algorithm yields good performance, regardless of the continuity of the coefficients. 
    
\begin{figure}[h!]
\centering
\begin{subfigure}[b]{0.3\textwidth}
\includegraphics[width=\textwidth]{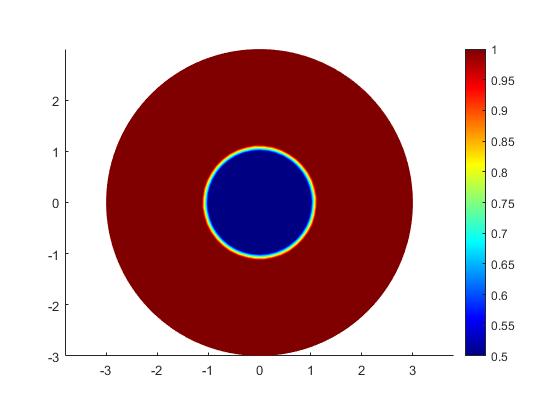}
\caption{Initial image $f_1$}
\end{subfigure}
\begin{subfigure}[b]{0.3\textwidth}
\includegraphics[width=\textwidth]{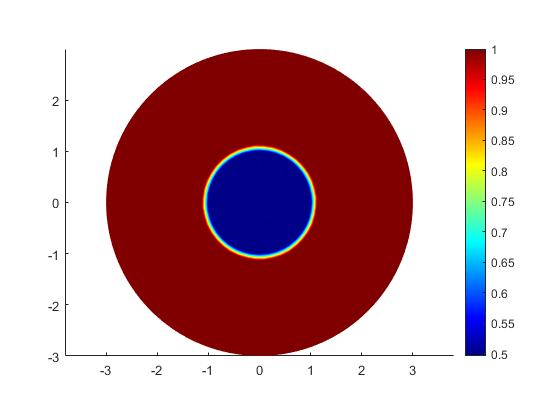}
\caption{Reconstruction with $c_1(|\xx|)$}
\label{fig:recon1_con}
\end{subfigure}
\begin{subfigure}[b]{0.3\textwidth}
\includegraphics[width=\textwidth]{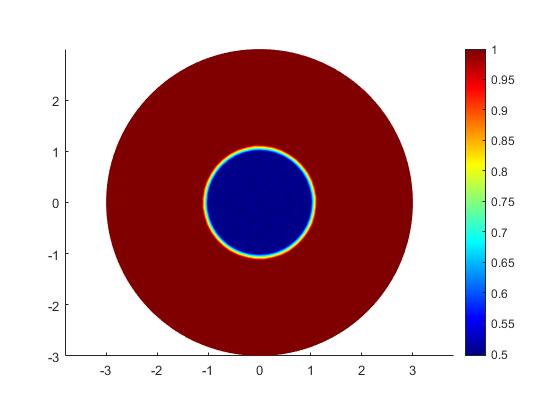}
\caption{Reconstruction with $c_2(|\xx|)$}
\label{fig:recon1_dis}
\end{subfigure}
\begin{subfigure}[b]{0.3\textwidth}
\includegraphics[width=\textwidth]{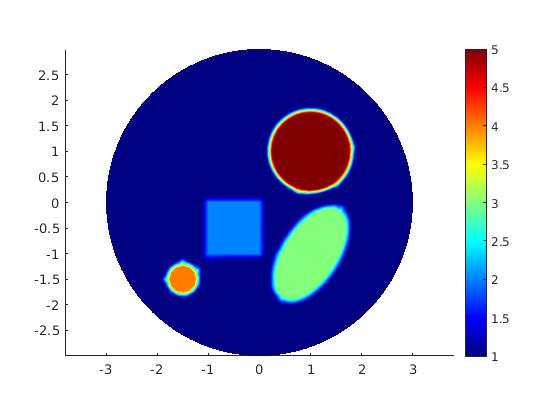}
\caption{Initial image $f_2$}
\end{subfigure}
\begin{subfigure}[b]{0.3\textwidth}
\includegraphics[width=\textwidth]{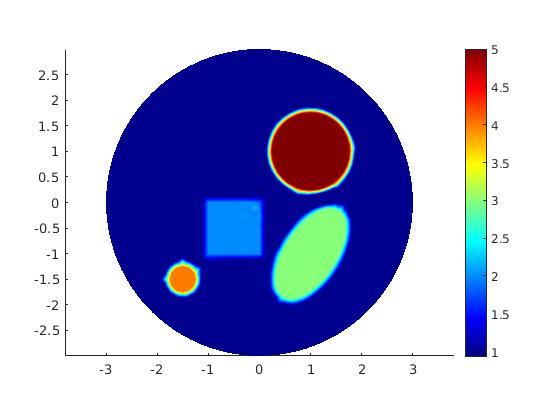}
\caption{Reconstruction with $c_1(|\xx|)$}
\label{fig:recon2_con}
\end{subfigure}
\begin{subfigure}[b]{0.3\textwidth}
\includegraphics[width=\textwidth]{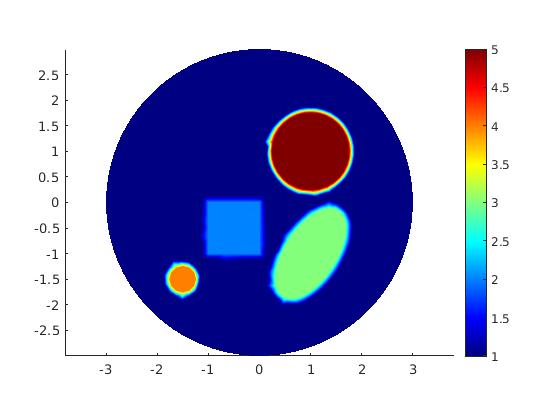}
\caption{Reconstruction with $c_2(|\xx|)$}
\label{fig:recon2_dis}
\end{subfigure}
\caption{Reconstruction results with radial variable coefficients.}
\label{fig:recon}
\end{figure}

%%%%%%%%%%%%%%%%%%%%%%%%%%%%%%%%%%%%%%%%%%%%%%%%
\section{Conclusion and discussion}\label{sec:discuss}
%%%%%%%%%%%%%%%%%%%%%%%%%%%%%%%%%%%%%%%%%%%%%
In this study, we determine the SVD of the wave forward operator with the Neumann boundary condition and radial variable speed.
To obtain the SVD, we first find an orthonormal basis $L^2_{}(B,c(|\xx|)^{-1})$ consisting of the eigenfunctions of 
 $ c(|\xx|)\triangle_\xx\phi(\xx)+\mu^2_k\phi(\xx)=0$ and $\partial_\nu \phi|_{S^{n-1}}=0.$
 Subsequently, after obtaining $\mathcal W \phi$, we show that $\mathcal W \phi$ is orthogonal.

This strategy can be applied to the wave forward operator with the Dirichlet boundary condition and radial variable speed.
%In \eqref{eq:pdeofpatorgin}, we can give the Dirichlet boundary condition assumption instead of the Neumann boundary condition (of course, we assume that $c(\xx)$ is also radial).
Specifically, the wave forward operator {\color{red} $\mathcal W_D$} satisfies the wave equations with $\mathcal W_Df(\xx ,0)=f(\xx )$, $\partial _t \mathcal W_Df(\xx ,0)=0$, and $\mathcal W_Df(\ttheta,t)=0.$ 
Finally, when the data can be represented as $\partial_\nu\mathcal W_Df(\ttheta,t)$, the SVD of $\partial_\nu\mathcal W_D$ for $n=2$ can be obtained using the following steps.
\begin{enumerate}
\item It can be shown that there exist eigenfunctions $h_{k,l,D}(r)$ of 
\begin{eqnarray*}
&\displaystyle\frac{c(r)}r\left(-\frac{\rm d}{{\rm d}r}\left(r\frac{\rm d}{{\rm d}r}\right)+\frac{l^2}r\right)h(r)=\mu_{k,l,D}h(r) \,\, \textrm{ with } \,\, h(1)=0
\end{eqnarray*}
(with eigenvalues $\mu_{k,l,D}$) such that they form an orthonormal basis of $L^2_{}((0,1],rc(r)^{-1})$.
\item For $\phi_{k,l,D}=h_{k,l}(|\xx|)e^{\mathrm{i}l\theta_\xx}/\sqrt{2\pi}$, $\mathcal W\phi_{k,l,D}=\frac{\rm d}{{\rm d}r}h_{k,l,D}(1)\Psi_{l,\mu_{k,l,D}}$. 
\item For $f\in L^2_{}(B^1,c(|\xx|)^{-1})$, we obtain
$$
\mathcal W_{D}f(\xx,t)=\sum_{k,l} \langle f,\phi_{k,l,D}\rangle \phi_{k,l,D}(\xx)\cos(\mu_{k,l,D}t).
$$ 
\item $\{\phi_{k,l,D},\Psi_{l,\mu_{k,l,D}},\frac{\rm d}{{\rm d}r}h_{k,l,D}(1)\}$ is an SVD of $\mathcal W_{D}\cdot|_{S^{1}\times[0,\infty)}:L^2_{}(B^1,c(|\xx|)^{-1})\to H(S^1\times[0,\infty))$.
%\item for $f\in \operatorname{span}\{\phi_{k,l,D}:k\in\mathbb N, l\in\mathbb Z\}$ and $(k,l)\in\mathbb N\times \mathbb Z $ with $<f,\phi_{k,l,D}>\neq 0$, 
%\begin{equation*}%\label{eq:main}
%<f,\phi_{k,l}>=%\displaystyle\intL_{ B} [\lambda_k^{-1}\overline{\phi_k(\uu)}\sin(\lambda_k T)\partial _\nu \mathcal W_{}f(\xx,T)-\lambda_k^{-1}\overline{\phi_k(\uu)}\sin(\lambda_k T) \mathcal W_{}f(\xx,T)]c(\xx)^{-1}{\rm d}\xx\\
%\displaystyle\lim_{A\to\infty}\frac{1}{2\pi |h_{k,l}(1)|^2 A}\intL^{A}_0\intL_{S^1} \overline{\phi_{k,l}(\ttheta)}\cos(\mu_{k,l} t)\mathcal W_{}f(\ttheta,t) {\rm d}S(\ttheta){\rm d}t.
%\end{equation*}
\end{enumerate}
A similar procedure can be followed to obtain a similar result for $\partial_\nu\mathcal W_D$ for $n=3$ steps.

%%%%%%%%%%%%%%%%%%%%%
\appendix
\section*{Appendix A: Theory of Sturm--Liouville operators}
In this appendix, we show the method by which the theory of Sturm--Liouville operators is applied in this study, as mentioned previously. This discussion focuses on showing that the operators used herein are self-adjoint operators such that they are purely discrete; therefore, the eigenfunctions corresponding to each operator form an orthonormal basis of the pertinent weighted $L^2$-space.     

By rewriting \eqref{eq:ode}, we obtain  
\begin{equation}\label{eq:ode 2}
   -(rf'(r))'+\frac{l^2}{r} f(r)=\lambda w(r) f(r) \qquad \left(\textrm{here } ':=\frac{\rm d}{{\rm d}r}\right), 
\end{equation}
where $w(r)=r/c(r)$. 
For a fixed $l\in \mathbb Z$, an operator $\mathcal S$ is defined on the weighted $L^2$-space with the weight $w(r)$, $L^2_w((0,1]):=L^2((0,1], rc(r)^{-1})$ corresponding to \eqref{eq:ode}, as follows:
\begin{equation*}
    \mathcal S(f)(r):=-(rf'(r))'+\frac{l^2}r f(r).
\end{equation*}
Therefore, $\mathcal S$ is a Sturm--Liouville operator acting on $L^2_w((0,1])$, and \eqref{eq:ode 2} is its eigenvalue equation. To obtain a {\it self-adjoint} operator, $\mathcal S$ is examined using the theory of Sturm--Liouville operators \cite{zettl05}. Because $p(x):=1/r$ and $q(x):=l^2/r$ are $L^1-$integrable near $r=1$ but not near $r=0$; therefore, $r=1$ is a regular endpoint and $r=0$ is a singular endpoint. First, at the regular endpoint $r=1$, we must define a boundary condition. For example, a well-known condition is the separated condition, $\alpha f(1)+\beta f'(1)=0$. For the singular endpoint $r=0$, we must classify the endpoint. In other words, based on the Weyl theory of Sturm--Liouville operators, it must be checked whether the endpoint $r=0$ satisfies the limit-circle (LC) classification, which indicates that all solutions to \eqref{eq:ode 2} are in $L^2_w((0,\epsilon))$ for $\epsilon\in (0,1)$ or the limit-point (LP) classification, which indicates that there must be a solution to \eqref{eq:ode 2} that is not $L^2_w$-integrable near $r=0$ (additional details can be found in Theorem 7.7 in \cite{pryce93}). Regarding the LP/LC classification of the endpoints, it is known (referring to Theorem 7.2.2 in \cite{zettl05}) that it suffices to examine only one (complex) value of $\lambda$. For convenience, we perform this examination when $\lambda=0$ (i.e., $\mu=0$ in \eqref{eq:ode}). 

In the problem considered in our study, the LP/LC classification depends on $l$ (when $l=0$ or when $l\in\mathbb{N}$). Directly solving \eqref{eq:ode 2} when $\lambda=0$ yields the following two linearly independent solutions. 
\begin{enumerate}
    \item[(i)] (when $l=0$) $1$ and $\ln r$
    \item[(ii)] (when $l\in\mathbb{N}$) $r^{-l}$ and $r^l$
\end{enumerate}
Because of the assumptions $w(r)=r/c(r)$ and $0< c_m \leq c(r) \leq c_M<\infty$, the solutions $1$, $\ln r$, and $r^{l}$ are $L^2_w$-integrable near $r=0$, but $r^{-l}$ is not. Therefore, the endpoint $r=0$ is LC when $l=0$ and LP when $l\in\mathbb{N}$. This implies that in order to obtain self-adjoint operator $\mathcal S$, we need not apply a boundary condition at $r=0$ when $l\in\mathbb{N}$. However, we require a boundary condition (for example, $f(0)=1$) when $l=0$ (additional details can be found in Theorem 10.4.1 in \cite{zettl05}). Therefore, in either case, our operator $\mathcal S$ can be assumed to be a self-adjoint operator.

Moreover, all solutions to \eqref{eq:ode 2} ($1$, $\ln r$, $r^{-l}$, and $r^{l}$) are non-oscillating near $0$, that is, there is a  $\epsilon>0$ such that the solutions are nonzero in $(0,\epsilon)$. The endpoint $r=0$ is considered {\em non-oscillatory} according to the Sturm--Liouville theory. Furthermore, for an endpoint in the LC case and $p(r)=1/r>0$ in $(0,1]$, the LCO (limit-circle and oscillatory) or LCNO (limit-circle and non-oscillatory) classifications are independent of $\lambda$. (see Theorem 7.3.1 in \cite{zettl05}). This implies that the self-adjoint operator $\mathcal S$ has a purely discrete spectrum (i.e., it has eigenvalues only with finite multiplicities and no continuous spectrum). 
\begin{thm}[Theorem 7.5 in \cite{pryce93} or Theorem 10.6.1 or Theorem 10.12.1 in \cite{zettl05}]
If both endpoints are in the LC (or regular) classification, then the spectrum is purely discrete and without a lower bound. If at least one endpoint is on the LP classification, the spectrum is purely discrete if and only if each LP endpoint is non-oscillatory for one $\mu\in\mathbb{C}$ (therefore, for all $\mu$). 
\end{thm}

In general, any function in $L^2_w$-space can be represented in the integral form via solutions to \eqref{eq:ode 2}.
Because $\mathcal S$ is purely discrete, the spectral theorem states that the projection-valued measure corresponding to $\mathcal S$ only has a point spectrum, which implies that the integral representation for the $L^2_w$-functions becomes a series function, and the set of eigenfunctions to $\mathcal S$ is dense in $L^2_w((0,1])$. 
\begin{thm}[Combination of Theorem 3.10 and Section 3.6 in \cite{folland92} or Theorem XIII.64 in \cite{R&S78}]
 For every singular Sturm--Liouville problem on $(a,b)$ such that the spectrum is purely discrete, there is an orthonormal basis $\{\phi_n \}_{n=1}^{\infty}$ of $L^2_w(a,b)$ consisting of eigenfunctions. 
\end{thm}

In summary (boundary condition at $r=0$ depending on $l$ may or may not be applied), in the problem represented by \eqref{eq:ode}, we can obtain a self-adjoint operator that is purely discrete such that its eigenfunctions form an orthonormal basis of the weighted $L^2$-space $L^2_w((0,1])$. 

%\begin{rmk}\label{Remark 7}
{\bf Remarks }  
Similar to the 2D setting, this discussion also applies for $n\ge 3$. In fact, spherical harmonics provides us the following relation:
\begin{equation}\label{eq:ode 3}
   -(r^{n-1}f'(r))'+l(l+n-2)r^{n-3} f(r)=\lambda r^{n-1}c(r)^{-1} f(r),
\end{equation}
which is similar to \eqref{eq:ode 2}. Solving \eqref{eq:ode 3} directly with $\lambda=0$ yields the following two linearly independent solutions. 
\begin{enumerate}
    \item[(i)] (when $l=0$ or $l=2-n$) \, $1$ and $r^{2-n}$  
    \item[(ii)] (when $l\in\mathbb{Z}\setminus\{0,2-n\}$) \, $r^{l}$ and $r^{-l-n+2}$ 
\end{enumerate}
Because all these solutions are non-oscillatory near $0$, a similar argument for $n=2$ implies that $0$ is either in the LC or LPNO (LP non-oscillatory) classification($1$ is a regular endpoint). In both cases, the related operators are purely discrete; therefore, the eigenfunctions form an orthonormal basis of the weighted $L_w^2$-space with $w(r)=r^{n-1}c(r)^{-1}$.
%\end{rmk}

\section*{Acknowledgements}
This work was supported by the National Research Foundation of Korea, a grant funded by the Korean government (NRF-2022R1C1C1003464, NRF-2020R1F1A1A01072414, NRF-2019R1F1A1061300).
\bibliographystyle{plain}

\end{document}